\RequirePackage{fix-cm}
\documentclass[smallextended]{article}       
\usepackage{amsmath}
\usepackage{amssymb,amsfonts}
\usepackage{graphicx}
\usepackage{times,amssymb,amscd}
\usepackage[USenglish]{babel}
\usepackage{verbatim}
\usepackage{enumerate}

\usepackage{epstopdf}

\newcommand{\bh}{\mathbf{h}}
\newcommand{\ba}{\mathbf{a}}

\newcommand{\cD}{\mathcal{D}}
\newcommand{\cF}{\mathcal{F}}

\newcommand{\cT}{\mathcal{T}}
\newcommand{\cC}{\mathcal{C}}
\newcommand{\cB}{\mathcal{B}}

\usepackage[usenames,dvipsnames]{xcolor} 
\usepackage{tikz, tikz-3dplot, pgfplots}
\usepackage{tkz-graph}
\usetikzlibrary[positioning,patterns] 

\usepackage{pictexwd,dcpic}

\usepackage{lscape}

\usepackage{amsthm}

\newtheorem{theorem}{Theorem}[section]
\newtheorem{corollary}[theorem]{Corollary}
\newtheorem{lemma}[theorem]{Lemma}
\newtheorem{exmple}[theorem]{Example}
\newtheorem{defn}[theorem]{Definition}
\newtheorem{rmrk}[theorem]{Remark}
\newtheorem{proposition}[theorem]{Proposition}

\newenvironment{remark}{\begin{rmrk}\normalfont}{\end{rmrk}}

\begin{document}

\author{Robert Thijs Kozma    \and    Jen\H{o}  Szirmai
}

\title
{New Lower Bound for the Optimal Ball Packing Density of Hyperbolic 4-space 
\footnote{Mathematics Subject Classification 2010: 52C17, 52C22, 52B15. \newline
Keywords and phrases: Coxeter group, hyperbolic geometry, packing, simplex tiling.}}

\author{\medbreak \medbreak {\normalsize{}} \\
\normalsize Robert Thijs Kozma$^{(1)}$~and Jen\H{o}  Szirmai$^{(2)}$\\
~\\
\normalsize (1) Department of Mathematics\\
\normalsize SUNY Stony Brook\\
\normalsize Stony Brook, NY 11794-3651 USA\\
\normalsize Email: rkozma@math.sunysb.edu \\
~\\
\normalsize (2) Budapest University of Technology and Economics\\
\normalsize Institute of Mathematics, Department of Geometry \\
\normalsize H-1521 Budapest, Hungary \\
\normalsize Email: szirmai@math.bme.hu }
\date{\normalsize (\today)}

\date{8/21/2014}

\maketitle

\begin{abstract}
In this paper we consider ball packings in $4$-dimensional hyperbolic
space. We show that it is possible to exceed the conjectured $4$-dimensional
realizable packing density upper bound due to L. Fejes-T\'oth
(Regular Figures, 1964). We give seven examples of horoball
packing configurations that yield higher densities of $0.71644896\dots$,
where horoballs are centered at ideal
vertices of certain Coxeter simplices, and are invariant under the actions of their respective  
Coxeter groups.

\end{abstract}


\section{Introduction}

Let $X$ denote a space of constant curvature, either the $n$-dimensional sphere $\mathbb{S}^n$, 
Euclidean space $\mathbb{E}^n$, or 
hyperbolic space $\mathbb{H}^n$ with $n \geq 2$. In discrete geometry, 
it is commonly asked to find the highest possible packing density in $X$ by congruent non-overlapping balls of a given radius \cite{Be}, \cite{G--K}. 
Euclidean cases are the best explored. For example,
the densest possible lattice packings are known for $\mathbb{E}^2$ through $\mathbb{E}^8$. In higher dimensions, however, mostly only bounds are known.
Furthermore, 
no sharp bounds exist for irregular packings in $\mathbb{E}^n$ when $n>3$.  
One major recent development has been the settling of the long-standing Kepler conjecture, 
part of Hilbert's 18th problem, by Thomas Hales at the turn of the 21st century.
Hales' computer-assisted proof was largely based on a program set forth by L. Fejes T\'oth in the 1950's \cite{Ha}.

The definition of packing density is critical in hyperbolic space as shown by B\"or\"oczky \cite{B78}. For other standard examples see also \cite{G--K}, \cite{R06}. 
The most widely accepted notion of packing density considers the local densities of balls with respect to their Dirichlet--Voronoi cells (cf. \cite{B78} and \cite{K98}). In order to consider horoball packings in $\overline{\mathbb{H}}^n$, we use an extended notion of such local density. 

Let $B$ be a horoball in packing $\cB$, and $P \in \overline{\mathbb{H}}^n$ be an arbitrary point. 
Define $d(P,B)$ to be the perpendicular distance from point $P$ to the horosphere $S = \partial B$, where $d(P,B)$ 
is taken to be negative when $P \in B$. The Dirichlet--Voronoi cell $\cD(B,\cB)$ of a horoball $B$ is defined as the convex body
\begin{equation}
\cD(B,\cB) = \{ P \in \mathbb{H}^n | d(P,B) \le d(P,B'), ~ \forall B' \in \cB \}. \notag
\end{equation}
Both $B$ and $\cD$ are of infinite volume, so the usual notion of local density is
modified as follows. Let $Q \in \partial{\mathbb{H}}^n$ denote the ideal center of $B$ at infinity, and take its boundary $S$ to be the one-point compactification of Euclidean $(n - 1)$-space.
Let $B_C^{n-1}(r) \subset S$ be the Euclidean $(n-1)$-ball with center $C \in S \setminus \{Q\}$.
Then $Q \in \partial {\mathbb{H}^n}$ and $B_C^{n-1}(r)$ determine a convex cone 
$\cC^n(r) = cone_Q\left(B_C^{n-1}(r)\right) \in \overline{\mathbb{H}}^n$ with
apex $Q$ consisting of all hyperbolic geodesics passing through $B_C^{n-1}(r)$ with limit point $Q$. The local density $\delta_n(B, \cB)$ of $B$ to $\cD$ is defined as
\begin{equation}
\delta_n(\cB, B) =\varlimsup\limits_{r \rightarrow \infty} \frac{vol(B \cap \cC^n(r))} {vol(\cD \cap \cC^n(r))}. \notag
\end{equation}
This limit is independent of the choice of center $C$ for $B^{n-1}_C(r)$.

In the case of periodic ball or horoball packings, the local density defined above can be extended to the entire hyperbolic space. This local density
is related to the simplicial density function (defined below) that we generalized in \cite{Sz12} and \cite{Sz12-2}.
In this paper we will use such definition of packing density (cf. Section 3).  

The alternate method suggested by Bowen and Radin \cite{Bo--R}, \cite{R06} uses Nevo's 
point-wise ergodic theorem 
to assure that the standard Euclidean limit notion of density is 
well-defined for $\mathbb{H}^n$. First they define a metric on 
the space $\Sigma_{\mathcal{P}}$ of relatively-dense packings by compact objects, based on Hausdorff distance, corresponding to uniform convergence on compact subsets of $\mathbb{H}^n$. Then they study the measures invariant under isometries of $\Sigma_{\mathcal{P}}$ rather than individual packings. There is a large class of packings of compact objects in hyperbolic space for which such density is well-defined. 
Using ergodic methods, 
they show that if there is only one optimally dense packing of $\mathbb{E}^n$ or $\mathbb{H}^n$, up to congruence, by congruent copies of bodies from some fixed finite collection, then that packing must have a symmetry group with compact fundamental domain. Moreover, for almost any radius $r \in [0,\infty)$ the optimal ball packing in $\mathbb{H}^n$ has low symmetry. 

A Coxeter simplex is an $n$-dimensional simplex in $X$ such that 
its dihedral angles are either submultiples of $\pi$, or zero. 
The group generated by reflections on the sides of a Coxeter simplex is called a Coxeter simplex reflection group. 
Such reflections give a discrete group of isometries of $X$ with the Coxeter simplex as its fundamental domain; 
hence the groups give regular tessellations of $X$. The Coxeter groups are finite for $\mathbb{S}^n$, and infinite for $\mathbb{E}^n$ or $\mathbb{H}^n$. 

In $\mathbb{H}^n$ we allow unbounded simplices with ideal vertices at infinity $\partial \mathbb{H}^n$. 
Coxeter simplices exist only for dimensions $n = 2, 3, \dots, 9$; furthermore, only a finite number exist in dimensions $n \geq 3$. 
Johnson {\it et al.} \cite{JKRT} computed the volumes of all Coxeter simplices in hyperbolic $n$-space, see also Kellerhals \cite{K91}. 
Such simplices are the most elementary building blocks of hyperbolic manifolds,
the volume of which is an important topological invariant. 

In the $n$-dimensional space $X$ of constant curvature
 $(n \geq 2)$, define the simplicial density function $d_n(r)$ to be the density of $n+1$ spheres
of radius $r$ mutually touching one another with respect to the simplex spanned by the centers of the spheres. L.~Fejes T\'oth and H.~S.~M.~Coxeter
conjectured that the packing density of balls of radius $r$ in $X$ cannot exceed $d_n(r)$.
Rogers \cite{Ro64} proved this conjecture in Euclidean space $\mathbb{E}^n$.
The $2$-dimensional spherical case was settled by L.~Fejes T\'oth \cite{FTL}, and B\"or\"oczky \cite{B78}, who proved the following extension:
\begin{theorem}[K.~B\"or\"oczky]
In an $n$-dimensional space of constant curvature, consider a packing of spheres of radius $r$.
In the case of spherical space, assume that $r<\frac{\pi}{4}$.
Then the density of each sphere in its Dirichlet--Voronoi cell cannot exceed the density of $n+1$ spheres of radius $r$ mutually
touching one another with respect to the simplex spanned by their centers.
\end{theorem}
In hyperbolic space, 
the monotonicity of $d_3(r)$ was proved by B\"or\"oczky and Florian
\cite{B--F64}; Marshall \cite{Ma99} 
showed that for sufficiently large $n$, 
function $d_n(r)$ is strictly increasing in variable $r$. Kellerhals \cite{K98} showed $d_n(r)<d_{n-1}(r)$, and that in cases considered by Marshall the local density of each ball in its Dirichlet--Voronoi cell is bounded above by the simplicial horoball density $d_n(\infty)$.

This upper bound for density in hyperbolic space $\mathbb{H}^3$ is $0.85327613\dots$, 
which is not realized by packing regular balls. However, it is attained by a horoball packing of
$\overline{\mathbb{H}}^3$ where the ideal centers of horoballs lie on an
absolute figure of $\overline{\mathbb{H}}^3$;
for example, they may lie at the vertices of the ideal regular
simplex tiling with Coxeter-Schl\"afli symbol $(3,3,6)$. 

In \cite{KSz} we proved that the optimal ball packing arrangement in $\mathbb{H}^3$ mentioned above is not unique. We gave several new examples of horoball packing arrangements based on totally asymptotic Coxeter tilings that yield the B\"or\"oczky--Florian upper bound \cite{B--F64}.
 
Furthermore, in \cite{Sz12}, \cite{Sz12-2} we found that 
by allowing horoballs of different types at each vertex of a totally asymptotic simplex and generalizing 
the simplicial density function to $\mathbb{H}^n$ for $(n \ge 2)$,
 the B\"or\"oczky-type density 
upper bound is no longer valid for the fully asymptotic simplices for $n \ge 3$. 
For example, in $\overline{\mathbb{H}}^4$ the locally optimal packing density is $0.77038\dots$, higher than the B\"or\"oczky-type density upper bound of $0.73046\dots$. 
However these ball packing configurations are only locally optimal and cannot be extended to the entirety of the
hyperbolic spaces $\mathbb{H}^n$. Further open problems and conjectures on $4$-dimensional hyperbolic packings are discussed in \cite{G--K--K}. 
A recent result of Jacquemet \cite{J} gives a formula for the inradius of a hyperbolic truncated $n$-simplex based on its Gram matrix.

The second-named author has several additional results on globally and locally optimal ball packings 
in $\mathbb{H}^n$, $\mathbb{S}^n$, and 
the eight Thurston geomerties arising from Thurston's geometrization conjecture 
\cite{Sz05-2}, \cite{Sz07-1}, \cite{Sz07-2}, \cite{Sz10}, \cite{Sz13-1}, \cite{Sz13-2}. 
These packing densities are global or local,
depending on whether the density obtained in a fundamental domain can or cannot be extended to the entire space.
 
In this paper we continue our investigations 
on ball packings in hyperbolic 4-space. 
Using horoball packings, allowing horoballs of different types,
we find seven counterexamples (realized by allowing up to three horoball types) 
to one of L. Fejes T\'oth's conjectures stated in the concluding section of his book Regular Figures:

\begin{quote}
Finally we draw attention to the tessellations $\{5,3,3,3\}$ of 4-dimensional hyperbolic space, the cell-inspheres and cell-circumspheres of which are also expected to form a closest packing and loosest covering. 
The corresponding densities are $(5-\sqrt{5})/4=0.690\dots$ and $(4+6\sqrt{5})/\sqrt{125}=1.557\dots$ \cite{FTL}
\end{quote}

\section{Higher Dimensional Hyperbolic Geometry}

In this paper we use the Cayley--Klein ball model, and a projective interpretation of hyperbolic geometry. This has the advantage of greatly 
simplifying our calculations in higher dimensions as compared to other models such as the Poincar\'e model. 
In this section we give a brief review of the concepts used in this paper. For a general discussion and background in hyperbolic geometry 
and the projective models of Thurston geometries see \cite{Mol97} and \cite{MSz}.

\subsection{The Projective Model}

We use the projective model in Lorentzian $(n+1)$-space
$\mathbb{E}^{1,n}$ of signature $(1,n)$, i.e.~$\mathbb{E}^{1,n}$ is
the real vector space $\mathbf{V}^{n+1}$ equipped with the bilinear
form of signature $(1,n)$
\begin{equation}
\langle ~ \mathbf{x},~\mathbf{y} \rangle = -x^0y^0+x^1y^1+ \dots + x^n y^n \label{bilinear_form}
\end{equation}
where the non-zero real vectors 
$\mathbf{x}=(x^0,x^1,\dots,x^n)\in\mathbf{V}^{n+1}$ 
and $ \mathbf{y}=(y^0,y^1,\dots,y^n)$ $\in\mathbf{V}^{n+1}$ represent points in projective space 
$\mathcal{P}^n(\mathbb{R})$. $\mathbb{H}^n$ is represented as the
interior of the absolute quadratic form
\begin{equation}
Q=\{[\mathbf{x}]\in\mathcal{P}^n | \langle ~ \mathbf{x},~\mathbf{x} \rangle =0 \}=\partial \mathbb{H}^n
\end{equation}
in real projective space $\mathcal{P}^n(\mathbf{V}^{n+1},
\mbox{\boldmath$V$}\!_{n+1})$. All proper interior points $\mathbf{x} \in \mathbb{H}^n$ satisfy
$\langle ~ \mathbf{x},~\mathbf{x} \rangle < 0$.

The boundary points $\partial \mathbb{H}^n $ in
$\mathcal{P}^n$ represent the absolute points at infinity of $\mathbb{H}^n$.
Points $\mathbf{y}$ satisfying $\langle ~ \mathbf{y},~\mathbf{y} \rangle >
0$ lie outside $\partial \mathbb{H}^n $ and are referred to as outer points
of $\mathbb{H}^n$. Take $P([\mathbf{x}]) \in \mathcal{P}^n$, point
$[\mathbf{y}] \in \mathcal{P}^n$ is said to be conjugate to
$[\mathbf{x}]$ relative to $Q$ when $\langle ~
\mathbf{x},~\mathbf{y} \rangle =0$. The set of all points conjugate
to $P([\mathbf{x}])$ form a projective (polar) hyperplane
\begin{equation}
pol(P)=\{[\mathbf{y}]\in\mathcal{P}^n | \langle ~ \mathbf{x},~\mathbf{y} \rangle =0 \}.
\end{equation}
Hence the bilinear form $Q$ in (\ref{bilinear_form}) induces a bijection
or linear polarity $\mathbf{V}^{n+1} \rightarrow
\mbox{\boldmath$V$}\!_{n+1}$
between the points of $\mathcal{P}^n$
and its hyperplanes.
Point $X [\bold{x}]$ and hyperplane $\alpha
[\mbox{\boldmath$a$}]$ are incident if the value of
the linear form $\mbox{\boldmath$a$}$ evaluated on vector $\bold{x}$ is
 zero, i.e. $\bold{x}\mbox{\boldmath$a$}=0$ where $\mathbf{x} \in \
\mathbf{V}^{n+1} \setminus \{\mathbf{0}\}$, and $\ \mbox{\boldmath$a$} \in
\mbox{\boldmath$V$}_{n
+1} \setminus \{\mbox{\boldmath$0$}\}$.
Similarly, lines in $\mathcal{P}^n$ are characterized by
2-subspaces of $\mathbf{V}^{n+1}$ or $(n-1)$-spaces of $\
\mbox{\boldmath$V$}\!_{n+1}$ \cite{Mol97}.

Let $P \subset \mathbb{H}^n$ denote a polyhedron bounded by
a finite set of hyperplanes $H^i$ with unit normal vectors
$\mbox{\boldmath$b$}^i \in \mbox{\boldmath$V$}\!_{n+1}$ directed
 towards the interior of $P$:
\begin{equation}
H^i=\{\mathbf{x} \in \mathbb{H}^d | \langle ~ \mathbf{x},~\mbox{\boldmath$b$}^i \rangle =0 \} \ \ \text{with} \ \
\langle \mbox{\boldmath$b$}^i,\mbox{\boldmath$b$}^i \rangle = 1.
\end{equation}
In this paper $P$ is assumed to be an acute-angled polyhedron
 with proper or ideal vertices.
The Grammian matrix $G(P)=( \langle \mbox{\boldmath$b$}^i,
\mbox{\boldmath$b$}^j \rangle )_{i,j} ~ {i,j \in \{ 0,1,2 \dots n \} }$  is an
indecomposable symmetric matrix of signature $(1,n)$ with entries
$\langle \mbox{\boldmath$b$}^i,\mbox{\boldmath$b$}^i \rangle = 1$
and $\langle \mbox{\boldmath$b$}^i,\mbox{\boldmath$b$}^j \rangle
\leq 0$ for $i \ne j$ where

$$
\langle \mbox{\boldmath$b$}^i,\mbox{\boldmath$b$}^j \rangle =
\left\{
\begin{aligned}
&0 & &\text{if}~H^i \perp H^j,\\
&-\cos{\alpha^{ij}} & &\text{if}~H^i,H^j ~ \text{intersect \ along an edge of $P$ \ at \ angle} \ \alpha^{ij}, \\
&-1 & &\text{if}~\ H^i,H^j ~ \text{are parallel in the hyperbolic sense}, \\
&-\cosh{l^{ij}} & &\text{if}~H^i,H^j ~ \text{admit a common perpendicular of length} \ l^{ij}.
\end{aligned}
\right.
$$
This is visualized using the weighted graph or scheme of the polytope $\sum(P)$. The graph nodes correspond 
to the hyperplanes $H^i$ and are connected if $H^i$ and $H^j$ not perpendicular ($i \neq j$).
If they are connected we write the positive weight $k$ where  $\alpha_{ij} = \pi / k$ on the edge, and  
unlabeled edges denote an angle of $\pi/3$. For examples, see the Coxeter diagrams in Table \ref{table:simplex_list}.

In this paper we set the sectional curvature of $\mathbb{H}^n$,
$K=-k^2$, to be $k=1$. The distance $d$ of two proper points
$[\mathbf{x}]$ and $[\mathbf{y}]$ is calculated by the formula
\begin{equation}
\cosh{{d}}=\frac{-\langle ~ \mathbf{x},~\mathbf{y} \rangle }{\sqrt{\langle ~ \mathbf{x},~\mathbf{x} \rangle
\langle ~ \mathbf{y},~\mathbf{y} \rangle }}.
\end{equation}
The perpendicular foot $Y[\mathbf{y}]$ of point $X[\mathbf{x}]$ dropped onto plane $[\mbox{\boldmath$u$}]$ is given by
\begin{equation}
\mathbf{y} = \mathbf{x} - \frac{\langle \mathbf{x}, \mathbf{u} \rangle}{\langle \mathbf{u}, \mathbf{u} \rangle} \mathbf{u},
\end{equation}
where $\mathbf{u}$ is the pole of the plane $[\mbox{\boldmath$u$}]$.

\subsection{Horospheres and Horoballs in $\mathbb{H}^n$}

A horosphere in $\mathbb{H}^n$ ($n \ge 2)$ is a 
hyperbolic $n$-sphere with infinite radius centered 
at an ideal point on $\partial \mathbb{H}^n$. Equivalently, a horosphere is an $(n-1)$-surface orthogonal to
the set of parallel straight lines passing through a point of the absolute quadratic surface. 
A horoball is a horosphere together with its interior. 

In order to derive the equation of a horosphere, we introduce a projective 
coordinate system for $\mathcal{P}^n$ with a vector basis 
$\bold{a}_i \ (i=0,1,2,\dots, n)$ so that the Cayley-Klein ball model of $\mathbb{H}^n$ 
is centered at $(1,0,0,\dots, 0)$, and set an
arbitrary point at infinity to lie at $A_0=(1,0,\dots, 0,1)$. 
The equation of a horosphere with center
$A_0$ passing through point $S=(1,0,\dots,0,s)$ is derived from the equation of the 
the absolute sphere $-x^0 x^0 +x^1 x^1+x^2 x^2+\dots + x^n x^n = 0$, and the plane $x^0-x^n=0$ tangent to the absolute sphere at $A_0$. 
The general equation of the horosphere is
\begin{equation}
0=\lambda (-x^0 x^0 +x^1 x^1+x^2 x^2+\dots + x^n x^n)+\mu{(x^0-x^n)}^2.
\end{equation}
Plugging in for $S$ we obtain
\begin{equation}
\lambda (-1+s^2)+\mu {(-1+s)}^2=0 \text{~~and~~} \frac{\lambda}{\mu}=\frac{1-s}{1+s}. \notag
\end{equation}
If $s \neq \pm1$, the equation of a horosphere in projective coordinates is
\begin{align}
(s-1)\left(-x^0 x^0 +\sum_{i=1}^n (x^i)^2\right)-(1+s){(x^0-x^n)}^2 & =0,
\end{align}
 and in cartesian coordinates setting $h_i=\frac{x^i}{x^0}$ it becomes
\begin{equation}  
\label{eqn:horosphere1}
\frac{2 \left(\sum_{i=1}^n h_i^2 \right)}{1-s}+\frac{4 \left(h_d-\frac{s+1}{2}\right)^2}{(1-s)^2}=1.
\end{equation}

In an $n$-dimensional hyperbolic space any two horoballs are congruent in the classical sense: each have an infinite radius.
However, it is often useful to distinguish between certain horoballs of a packing. 
We use the notion of horoball type with respect to the packing as introduced in \cite{Sz12-2}.

Two horoballs of a horoball packing are said to be of the {\it same type} or {\it equipacked} if 
and only if their local packing densities with respect to a given cell (in our case a 
Coxeter simplex) are equal. If this is not the case, then we say the two horoballs are of {\it different type}. 
For example, in the above discussion horoballs centered at $A_0$ passing through $S$ with different values for the final coordinate $s$ are of different type relative to 
an appropriate cell.

In order to compute volumes of horoball pieces, we use J\'anos Bolyai's classical formulas from the mid 19-th century:
\begin{enumerate}
\item 
The hyperbolic length $L(x)$ of a horospheric arc that belongs to a chord segment of length $x$ is
\begin{equation}
\label{eq:horo_dist}
L(x)=2 \sinh{\left(\frac{x}{2}\right)} .
\end{equation}
\item The intrinsic geometry of a horosphere is Euclidean, 
so the $(n-1)$-dimensional volume $\mathcal{A}$ of a polyhedron $A$ on the 
surface of the horosphere can be calculated as in $\mathbb{E}^{n-1}$.
The volume of the horoball piece $\mathcal{H}(A)$ determined by $A$ and 
the aggregate of axes 
drawn from $A$ to the center of the horoball is
\begin{equation}
\label{eq:bolyai}
vol(\mathcal{H}(A)) = \frac{1}{n-1}\mathcal{A}.
\end{equation}
\end{enumerate}
\section{Horoball packings of Coxeter Simplices with Ideal Verticies}

Let $\cT$ be a Coxeter tiling. A rigid motion mapping one cell of $\cT$ onto
another maps the entire tiling onto itself. The symmetry group of a Coxeter tiling
 is its Coxeter group, denoted by $\Gamma_\cT$. 
Any simplex cell of $\cT$ acts as a fundamental domain $\cF_{\cT}$
of $\Gamma_\cT$, where the Coxeter group is generated by reflections on the $(n - 1)$-dimensional facets of $\cF_{\cT}$. 
In this paper we consider only asymptotic Coxter simplices, i.e. ones that have at least one ideal vertex. In Table \ref{table:simplex_list} we list the nine asymptotic Coxeter simplices that exist in hyperbolic 4-space, together with their volumes. For a complete discussion of hyperbolic Coxeter 
simplices 
and their volumes for dimensions $n \geq 3$, see Johnson {\it et al.} \cite{JKRT}. 

We define the density of a horoball packing $\mathcal{B}_{\cT}$ of a Coxeter simplex tiling $\cT$ as
\begin{equation}
\delta(\mathcal{B}_{\cT})=\frac{\sum_{i=1}^n vol(\mathcal{B}_i \cap \cF_{\cT})}{vol(\cF_{\cT})}.
\end{equation}
Here $\cF_{\cT}$ denotes the simplicial fundamental domain of tiling $\cT$, $n$ is the number of ideal vertices of $\cF_\cT$, 
and $\mathcal{B}_i$ are the horoballs centered at ideal vertices. 
We allow horoballs of different types at the asymptotic vertices of the tiling. 
A horoball type is allowed if it yields a packing: no two horoballs may have an interior point in common. 
In addition we require that no horoball extend beyond the facet opposite the vertex where it is centered so that the packing remains invariant under the actions of the Coxeter group of the tiling.
If these conditions are satisfied, we can extend the packing density from 
the simplicial fundamental domain $\cF_{\cT}$ to the entire $\mathbb{H}^4$ using the Coxeter group $\cT$ associated with a tiling.
In the case of Coxeter simplex tilings, Dirichlet--Voronoi cells coincide with the Coxeter simplices. We denote the optimal horoball packing density as
\begin{equation}
\delta_{opt}(\cT) = \sup\limits_{\mathcal{B}_{\cT} \text{~packing}} \delta(\mathcal{B}_{\cT}).
\end{equation}

\begin{table}
    \begin{tabular}{|cc|c|c|c|}
    \hline
    Coxeter & ~ & Witt & Simplex & Packing\\
    Diagram  & Notation & Symbol & Volume &Density\\
    \hline
        Simply Asymptotic & ~ & ~ & ~ & ~\\
    \hline
    \begin{tikzpicture}
	
	\draw (0,0) -- (1,0);
	\draw (1,0) -- (1.5,0.25);
	\draw (1,0) -- (1.5,-0.25);
	
	\draw[fill=black] (0,0) circle (.05);
	\draw[fill=black] (0.5,0) circle (.05);
	\draw[fill=black] (1,0) circle (.05);
	\draw[fill=black] (1.5,0.25) circle (.05);
	\draw[fill=black] (1.5,-0.25) circle (.05);
	
	\node at (1,-0.175) {$4$};
\end{tikzpicture}
  & $[4, 3^{2,1}]$ & $\overline{S}_4$ & $\pi^2/1440$ & 0.71644896 \\

 \begin{tikzpicture}

	\draw (0,0) -- (.5,0);
	\draw (.5,0) -- (1,0.25);
	\draw (.5,0) -- (1,-0.25);
	\draw (1,0.25) -- (1.5,0);
	\draw (1,-0.25) -- (1.5,0);
	
	\draw[fill=black] (0,0) circle(.05);
	\draw[fill=black] (.5,0) circle(.05);
	\draw[fill=black] (1,0.25) circle(.05);
	\draw[fill=black] (1,-0.25) circle(.05);
	\draw[fill=black] (1.5,0) circle(.05);

\end{tikzpicture}
  & $[3, 3^{[4]}]$ & $\overline{P}_4$ & $\pi^2/720$ & 0.71644896 \\

    \begin{tikzpicture}

	\draw (0,0) -- (2,0);
	
	\draw[fill=black] (0,0) circle(.05);
	\draw[fill=black] (0.5,0) circle(.05);
	\draw[fill=black] (1,0) circle(.05);
	\draw[fill=black] (1.5,0) circle(.05);
	\draw[fill=black] (2,0) circle(.05);

	\node at (.75,0.15) {$4$};
	\node at (1.75,0.15) {$4$};
	 
\end{tikzpicture}
  & $[3,4, 3, 4]$ & $\overline{R}_4$ & $\pi^2/864$ & 0.60792710 \\

    \begin{tikzpicture}
	
	\draw (0,0) -- (1,0);
	\draw (1,0) -- (1.5,0.25);
	\draw (1,0) -- (1.5,-0.25);
	
	\draw[fill=black] (0,0) circle (.05);
	\draw[fill=black] (.5,0) circle (.05); 
	\draw[fill=black] (1,0) circle (.05);
	\draw[fill=black] (1.5,0.25) circle (.05);
	\draw[fill=black] (1.5,-0.25) circle (.05);
	
	\node at (.75,0.175) {$4$};

\end{tikzpicture}

  & $[3, 4, 3 ^{1,1}]$ & $\overline{O}_4$ & $\pi^2/432$ & 0.60792710 \\

  \begin{tikzpicture}

	\draw (0,0) -- (1,0);
	\draw (0,0) -- (0.25,.5);
	\draw (0.25,.5) -- (.75,.5);
	\draw (.75,.5) -- (1,0);
	
	\draw[fill=black] (0,0) circle(.05);
	\draw[fill=black] (.5,0) circle(.05);
	\draw[fill=black] (1,0) circle(.05);
	\draw[fill=black] (0.25,.5) circle(.05);
	\draw[fill=black] (.75,.5) circle(.05);

	\node at (0,0.275) {$4$};
	\node at (1,0.275) {$4$};
 
\end{tikzpicture}
  & $[(3^2,4,3,4)]$ & $\widehat{FR}_4$ & $\pi^2/108$ & 0.71644896 \\
\hline
    Doubly Asymptotic & ~ & ~ & ~ & ~\\
\hline
    \begin{tikzpicture}
	
	\draw (0,0) -- (1,0);
	\draw (1,0) -- (1.5,0.25);
	\draw (1,0) -- (1.5,-0.25);
	
	\draw[fill=black] (0,0) circle (.05);
	\draw[fill=black] (.5,0) circle (.05);
	\draw[fill=black] (1,0) circle (.05);
	\draw[fill=black] (1.5,0.25) circle (.05);
	\draw[fill=black] (1.5,-0.25) circle (.05);
	
	\node at (1,-0.2) {$4$};
	\node at (0.25,0.15) {$4$};

\end{tikzpicture}

  & $[4,~^{3,4}_4]$ & $\overline{N}_4$ & $\pi^2/288$ & 0.71644896 \\

    \begin{tikzpicture}

	\draw (0,0) -- (.5,0);
	\draw (.5,0) -- (1,0.25);
	\draw (.5,0) -- (1,-0.25);
	\draw (1,0.25) -- (1.5,0);
	\draw (1,-0.25) -- (1.5,0);
	
	\draw[fill=black] (0,0) circle(.05);
	\draw[fill=black] (.5,0) circle(.05);
	\draw[fill=black] (1,0.25) circle(.05);
	\draw[fill=black] (1,-0.25) circle(.05);
	\draw[fill=black] (1.5,0) circle(.05);

	\node at (0.25,0.15) {$4$};
	 
\end{tikzpicture}
  & $[4, 3^{[4]}]$ & $\overline{BP}_4$ & $\pi^2/144$ & 0.71644896 \\
\hline
    Triply Asymptotic & ~ & ~ & ~ & ~\\
\hline
\begin{tikzpicture}
	
	\draw (0,0.25) -- (.5,0);
	\draw (0,-0.25) -- (.5,0);
	\draw (.5,0) -- (1, 0.25);
	\draw (.5,0) -- (1,-0.25);
	
	\draw[fill=black] (0,0.25) circle (.05);
	\draw[fill=black] (0,-0.25) circle (.05);
	\draw[fill=black] (.5,0) circle (.05);
	\draw[fill=black] (1,0.25) circle (.05);
	\draw[fill=black] (1,-0.25) circle (.05);
	
	\node at (.75,-0.3) {$4$};

\end{tikzpicture}
  & $[4, 3^{1,1,1}]$ & $\overline{M}_4$ & $\pi^2/144$ & 0.71644896 \\

  \begin{tikzpicture}

	\draw (0,0) -- (.5,0.25);
	\draw (0,0) -- (.5,-0.25);
	\draw (.5,0.25) -- (.5,0);
	\draw (.5,0.25) -- (1,0);
	\draw (.5,-0.25) -- (.5,0);
	\draw (.5,-0.25) -- (1,0);
	
	\draw[fill=black] (0,0) circle(.05);
	\draw[fill=black] (.5,0) circle(.05);
	\draw[fill=black] (1,0) circle(.05);
	\draw[fill=black] (.5,0.25) circle(.05);
	\draw[fill=black] (.5,-0.25) circle(.05);
 
\end{tikzpicture}
  & $[3^{[3]\times[]}]$ & $\overline{DP}_4$ & $\pi^2/72$ & 0.71644896 \\
\hline

    \end{tabular}
    \caption{Notation and volumes for the nine asymptotic Coxeter Simplices in $\mathbb{H}^4$.}
    \label{table:simplex_list}
\end{table}

\begin{figure}
\begin{center}
\includegraphics[height=60mm]{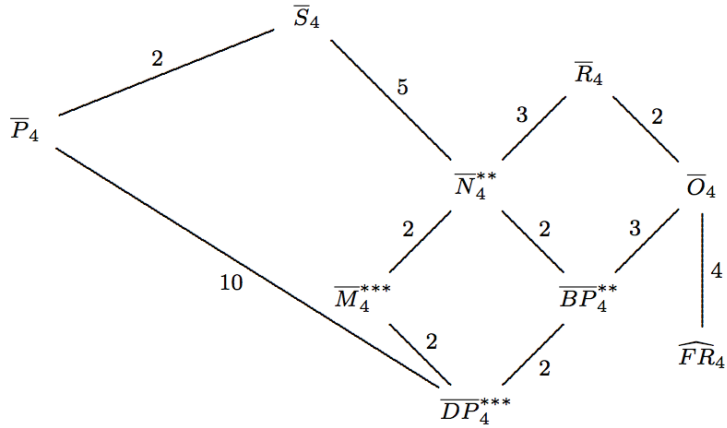}
\end{center}
\caption{Lattice of Subgroups of cocompact Coxeter groups in $\mathbb{H}^4$. The number of stars in the superscript $^{**}$ and $^{***}$ indicates that the fundamental simplex of the group has two or three ideal vertices.}
\label{fig:lattice_of_subgroups}
\end{figure}

The asymptotic Coxeter simplex tilings are related through the subgroup structure of their Coxeter symmetry groups as shown in 
Figure \ref{fig:lattice_of_subgroups} \cite{JKRT2}, \cite{JW}.
Let $\Gamma_1$ and $\Gamma_2$ be the two Coxeter symmetry groups of Coxeter tilings $\cT_1$ and $\cT_2$, respectively. 
When the index of Coxeter group $\Gamma_1$ in $\Gamma_2$ is two, 
i.e. $|\Gamma_1:\Gamma_2| = 2$, then the two Coxeter groups differ 
by one reflection, and the fundamental domain of $\Gamma_2$ is 
obtained from that of $\Gamma_1$ by domain doubling, 
that is by  merging a certain pair of neighboring domains by removing a common facet. 
In the case of asymptotic Coxeter simplices 
if $|\Gamma_1:\Gamma_2| = 2$ and the number of asymptotic vertices of the fundamental domains 
of $\Gamma_1$ and $\Gamma_2$ are equal, then the new fundamental domain is obtained 
by removing a facet adjacent to an asymptotic vertex. If the number of asymptotic vertices increases 
by one, then the cells of $\Gamma_2$ are obtained by removing a facet opposite to the asymptotic vertices 
of $\Gamma_1$ and merging the cells. The relationship between the volumes of the cells of 
$\Gamma_1$ and $\Gamma_2$ when $|\Gamma_1 : \Gamma_2|=m$ is given by 
$vol(\cF_{\cT_{1}})= m \cdot vol(\cF_{\cT_{2}})$. 
If the index of the groups is two, then a packing density 
$\delta(\cB_{\cT_{1}})$ for the bigger group can be extended to the smaller group $\Gamma_2$.

\subsection{Simply Asymptotic Cases}

We compute the optimal horoball packing density for the Coxeter simplex tiling $\overline{S}_4$; 
the other simply asymptotic cases can be obtained using the same method. Case 
$\overline{R}_4$ was computed by the second-named author in \cite{Sz05-2}. 

\begin{proposition}
\label{proposition:s4}
The optimal horoball packing density for simply asymptotic Coxeter simplex tiling $\cT_{{\overline{S}}_4}$ is $\delta_{opt}({\overline{S}}_4)
\approx 0.71644896$.
\end{proposition}

\begin{proof}
Let $\cF_{{\overline{S}_4}}$ be the simplicial fundamental domain of Coxeter tiling $\cT_{\overline{S}_4}$. 
We set coordinates for its vertices $A_0, A_1, \dots, A_4$ that satisfy the angle requirements. 
Our choice of vertices, as well as forms for hyperplanes $[\mbox{\boldmath$u$}_i]$ opposite to vertices $A_i$,  
are given in Table \ref{table:data_s_a}. 
In order to maximize the packing density, we determine the largest horoball type $\mathcal{B}_0(s)$ centered at ideal vertex $A_0$ that is admissible in cell $\cF_{{\overline{S}}_4}$. This is the horoball with type-parameter 
$s$ (intuitively the ``radius" of the horoball) such that the horoball $\mathcal{B}_0(s)$ is tangent to 
the plane of the hyperface $[\mbox{\boldmath$u$}_0]$ bounding 
the fundamental simplex opposite of $A_0$. The perpendicular foot $F_0[\mbox{\boldmath$f$}_0]$ of vertex $A_0$ on plane $[\mbox{\boldmath$u$}_0]$,
\begin{equation}
\mbox{\boldmath$f$}_0 =\ba_0 - \frac{\langle \ba_0, \mathbf{u}_0 \rangle}{\langle \mathbf{u}_0,\mathbf{u}_0 \rangle} 
\mathbf{u}_0 = \left(1,0,-\frac{2}{5},\frac{1}{5},0\right),
\end{equation}
is the point of tangency of horoball $\mathcal{B}_0(s)$ and hyperface $\mathbf{u}_0$ of the the simplex cell.

Plugging in for $F_0$ and solving equation (\ref{eqn:horosphere1}), we find that the horoball with type-parameter
 $s=-\frac{1}{9}$ is the optimal type. 
The equation of horosphere $\partial \mathcal{B}_0 =\partial \mathcal{B}_0(-\frac{1}{9})$ centered at $A_0$ passing through $F_0$ is

\begin{equation}
\frac{9}{5} \left(h_1^2+h_2^2+h_3^2\right)+\frac{81}{25} \left(h_4-\frac{4}{9}\right)^2=1.
\end{equation}

\begin{figure}[h]
\begin{center}
\includegraphics[height=50mm]{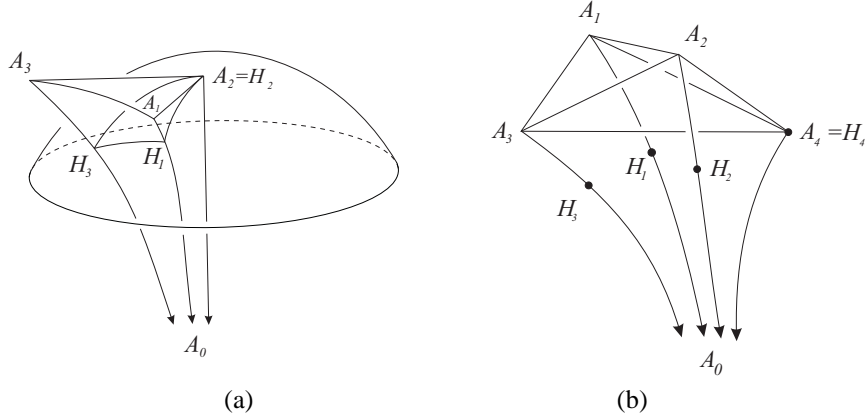}
(a)~~~~~~~~~~~~~~~~~~~~~~~~~~~~~~~~~~~~~~~~~~~~~~~~~~~~~~~(b)
\end{center}
\caption{Simply asymptotic case. (a) Horoball $\mathcal{B}_0$ intersecting the sides of the simplex at $H_1$, $H_2$, and $H_3$. (b) Horospheric tetrahedron on hyperface opposite $A_0$.}
\label{fig:horospheric_triangle}
\end{figure} 

The intersections $H_i[\bh_i]$ of horosphere $\partial \mathcal{B}_0$ and simplex edges are found by parameterizing the 
simplex edges as $\bh_i(\lambda) = \lambda \ba_0+\ba_i$ $(i=1,2,3,4)$, and computing their intersections with $\partial \mathcal{B}_0$. 
See Figure (\ref{fig:horospheric_triangle}), and Table \ref{table:data_s_a} for the intersection points.
The volume of the horospherical tetrahedron determines the volume of the horoball piece by equation \eqref{eq:bolyai}.
In order to determine the data of the horospheric tetrahedron, we compute the hyperbolic distances $l_{ij}$ by the formula (5)
$l_{ij} = d(H_i, H_j)$ where $d(\bh_i,\bh_j)= \arccos\left(\frac{-\langle \bh_i, \bh_j 
\rangle}{\sqrt{\langle \bh_i, \bh_i \rangle \langle \bh_j, \bh_j \rangle}}\right)$.
Moreover, the horospherical distances $L_{ij}$ can be calculated by formula \eqref{eq:horo_dist}.
The intrinsic geometry of the horosphere is Euclidean, so we use the 
Cayley-Menger determinant to find the volume $\mathcal{A}$ of the horospheric tetrahedron $A$,

\begin{equation}
\mathcal{A} = \frac{1}{288}
\begin{vmatrix}
 0 & 1 & 1 & 1 & 1 \\
 1 & 0 & L_{12}^2 & L_{13}^2 & L_{14}^2 \\
 1 & L_{12}^2 & 0 & L_{23}^2 & L_{24}^2 \\
 1 & L_{13}^2 & L_{23}^2 & 0 & L_{34}^2 \\
 1 & L_{14}^2 & L_{24}^2 & L_{34}^2 & 0
 \end{vmatrix} \approx 0.0147314.
 \end{equation}

The volume of the optimal horoball piece in the fundamental simplex is

\begin{equation}
vol(\mathcal{B}_0 \cap \cF_{\overline{S}_4}) = \frac{1}{n-1}\mathcal{A} \approx \frac{1}{3} \cdot 0.0147314 \approx 0.00491046.
\end{equation}

Hence by the Coxeter group $\Gamma_{\overline{S}_4}$ the optimal horoball packing density of the Coxeter Simplex tiling $\cT_{\overline{S}_4}$ becomes

\begin{equation}
\delta_{opt}(\overline{S}_4) = \frac{vol(\mathcal{B}_0 \cap \cF_{\overline{S}_4})}{vol(\cF_{\overline{S}_4})}\\
    \approx \frac{0.00491046}{\pi^2/1440} \\
    \approx 0.71644896. 
\end{equation}
\qed
\end{proof}

The same method is used to find the optimal packing density of the remaining simply asymptotic Coxeter simplex tilings. Results of the computations are given  
in Table \ref{table:data_s_a}. We summarize the results:

\begin{corollary}
The optimal horoball packing density for simply asymptotic Coxeter simplex tiling $\cT_\Gamma$, $\Gamma \in \Big\{ \overline{S}_4, 
\overline{P}_4, \widehat{FR}_4 \Big\}$ is $\delta_{opt}(\Gamma) \approx 0.71644896$.
\end{corollary}

\begin{landscape}
\begin{table}
	\begin{tabular}{|l|l|l|l|l|l|}
		 \hline
		 \multicolumn{6}{|c|}{{\bf Coxeter Simplex Tilings} }\\
		\hline
		 Witt Symb. & $\overline{S}_4$ &  $\overline{P}_4$ &  $\overline{R}_4$ &  $\overline{O}_4$ &  $\widehat{FR}_4$ \\
		 \hline
		 \multicolumn{6}{|c|}{{\bf Vertices of Simplex} }\\
		 \hline
		 $A_0$ & $(1,0,0,0,1)$ & $(1,0,0,0,1)$ & $(1,0,0,0,1)$ & $(1,0,0,0,1)$ & $(1,0,0,0,1)$ \\
		 $A_1$ & $(1,0,\frac{-2}{5},\frac{1}{5},0)$ &$(1,0,\frac{\sqrt{22}}{11},-\frac{\sqrt{11}}{11},0)$&$(1,\frac{1}{2},\frac{1}{2},\frac{1}{2},0)$&$(1,0,\frac{\sqrt{42}}{14},\frac{\sqrt{21}}{14},\frac{1}{2})$&$(1,0,\frac{\sqrt{15}}{5},0,\frac{1}{2})$ \\
		 $A_2$ & $(1,0,0,0,-\frac{1}{4})$ & $(1,0,0,0,-\frac{3}{8})$ & $(1,\frac{1}{2},\frac{1}{2},0,0)$ & $(1,0,\frac{\sqrt{42}}{11},0,\frac{4}{11})$ & $(1,0,\frac{\sqrt{15}}{8},-\frac{\sqrt{15}}{8},\frac{11}{16})$ \\
		 $A_3$ & $(1,0,\frac{1}{4},\frac{1}{2},-\frac{1}{4})$ & $(1,0,\frac{\sqrt{22}}{9},\frac{\sqrt{11}}{9},\frac{-2}{9})$ & $(1,\frac{1}{2},0,0,0)$ & $(1,0,0,0,-\frac{2}{5})$ & $(1,0,0,0,-\frac{1}{4})$ \\
		 $A_4$ & $(1,\frac{\sqrt{10}}{8},0,0,-\frac{1}{4})$ & $(1,\frac{\sqrt{22}}{10},\frac{\sqrt{22}}{10},0,-\frac{1}{10})$ & $(1,0,0,0,0)$ & $(1,\frac{\sqrt{42}}{11},\frac{\sqrt{42}}{11},0,\frac{4}{11})$ & $(1,\frac{\sqrt{15}}{5},0,0,\frac{1}{2})$ \\
		 \hline
		 \multicolumn{6}{|c|}{{\bf The form $\mbox{\boldmath$u$}_i$ of sides opposite $A_i$ }}\\
		\hline
		 $\mbox{\boldmath$u$}_0$ & $(1, 0, 2, -1, 4)^T$ & $(1,0,-\frac{\sqrt{22}}{3},\frac{\sqrt{11}}{3},\frac{8}{3})^T$ & $(0,0,0,0,1)^T$ & $(1,0,-\frac{\sqrt{42}}{2},-\frac{\sqrt{21}}{2},\frac{5}{2})^T$ & $(1,-\sqrt{15},-\sqrt{15},\sqrt{15},4)^T$ \\
		 $\mbox{\boldmath$u$}_1$ & $(0, 0, -2, 1, 0)^T$ & $(0,\frac{\sqrt{11}}{\sqrt{22}},-\frac{\sqrt{11}}{\sqrt{22}},1,0)^T$ & $(0,0,0,1,0)^T$ & $(0, 0, 0, 1, 0)^T$ & $(0, 0, 1, 1, 0)^T$ \\
		 $\mbox{\boldmath$u$}_2$ & $(1,-\sqrt{10},1,-3,-1)^T$ & $(1,0,-\frac{\sqrt{22}}{2},0,-1)^T$ & $(0,0,-1,1,0)^T$ & $(0,-\frac{1}{\sqrt{2}},\frac{1}{\sqrt{2}},-1,0)^T$ & $(0, 0, 0, -1, 0)^T$ \\
		 $\mbox{\boldmath$u$}_3$ & $(0,0,\frac{1}{2},1,0)^T$ & $(0,-\frac{\sqrt{11}}{\sqrt{22}},\frac{\sqrt{11}}{\sqrt{22}},1,0)^T$ & $(0,-1,1,0,0)^T$ & $(1,0,\frac{\sqrt{42}}{6},0,-1)^T$ & $(1,-\frac{\sqrt{15}}{6},-\frac{\sqrt{15}}{6},0,-1)^T$ \\
		 $\mbox{\boldmath$u$}_4$ & $(0, 1, 0, 0, 0)^T$ & $(0, 1, 0, 0, 0)^T$ & $(1,-2,0,0,-1)^T$ & $(0, -1, 0, 0, 0)^T$ & $(0, 1, 0, 0, 0)^T$ \\
		 \hline
		 \multicolumn{6}{|c|}{{\bf Maximal horoball parameter $s$ }}\\
		\hline
		 $s$ & $-1/9$ & $-3/19$ & $0$ & $5/19$ & $7/17$ \\
		\hline
		 \multicolumn{6}{|c|}{ {\bf Intersections $H_i = \mathcal{B}(A_0,s)\cap A_0A_i$ of horoballs with simplex edges}}\\
		\hline
		 $H_1$ & $(1,0,-\frac{2}{5},\frac{1}{5},0)$ & $(1,0,\sqrt{\frac{2}{11}},-\frac{1}{\sqrt{11}},0)$ & $(1,\frac{4}{11},\frac{4}{11},\frac{4}{11},\frac{3}{11})$ & $(1,0,\sqrt{\frac{3}{14}},\sqrt{\frac{3}{28}},\frac{1}{2})$ & $(1,0,0,0,\frac{7}{17})$ \\
		 $H_2$ &$(1,0,0,0,-\frac{1}{9})$ & $(1,0,0,0,-\frac{3}{19})$ & $(1,\frac{2}{5},\frac{2}{5},0,\frac{1}{5})$ & $(1,0,\frac{2 \sqrt{42}}{25},0,\frac{11}{25})$ & $(1,0,\frac{4 \sqrt{15}}{29},0,\frac{19}{29})$ \\
		 $H_3$ &$(1,0,\frac{1}{5},\frac{2}{5},0)$ & $(1,0,\sqrt{\frac{2}{11}},\frac{1}{\sqrt{11}},0)$ & $(1,\frac{4}{9},0,0,\frac{1}{9})$ & $(1,0,0,0,\frac{5}{19})$ & $(1,0,\frac{4 \sqrt{15}}{41},-\frac{4 \sqrt{15}}{41},\frac{31}{41})$ \\
		 $H_4$ &$(1,\frac{2 \sqrt{10}}{19},0,0,-\frac{1}{19})$ & $(1,\frac{2 \sqrt{22}}{23},\frac{2 \sqrt{22}}{23},0,\frac{1}{23})$ & $(1,0,0,0,0)$ & $(1,\frac{2 \sqrt{42}}{31},\frac{2 \sqrt{42}}{31},0,\frac{17}{31})$ & $(1,\frac{4 \sqrt{15}}{29},0,0,\frac{19}{29})$ \\
		 \hline
		 \multicolumn{6}{|c|}{ {\bf Volume of horoball pieces }}\\
		\hline
		 $vol(\mathcal{B}_0 \cap \mathcal{F})$ & $0.00491046$ & $0.00982093$ & $0.00694444$ & $0.01388889$ & $0.0555556$ \\
		\hline
		\multicolumn{6}{|c|}{ {\bf Optimal Packing Density}}\\
		\hline
		 $\delta_{opt}$ & $0.71644896$ & $0.71644896$ & $0.60792710$ & $0.60792710$ & $0.71644896$ \\
		\hline
	\end{tabular}
	\caption{Data for simply asymptotic Tilings in Cayley-Klein ball model of radius 1 centered at (1,0,0,0,0)}
	\label{table:data_s_a}
\end{table}
\end{landscape}

\subsection{Multiply Asymptotic Cases}

In cases where the Coxeter simplex has multiple asymptotic vertices,  
we allow horoballs of different types at each vertex. 
The equations of horospheres centered at $(1,0,0,$ $0,-1)$ and $(1,0,1,0,0)$ where $h_i=\frac{x^i}{x^0}$ are

\begin{equation}
\label{eq:horoMW}
\frac{2 \left(h_1^2+h_2^2+h_3^2\right)}{s+1}+\frac{4 \left(h_4+ \frac{1-s}{2}\right)^2}{(s+1)^2}=1,
\end{equation}
and
\begin{equation}
\label{eq:horoY}
\frac{2 \left(h_1^2+h_3^2+h_4^2\right)}{1-s}+\frac{4 \left(h_2-\frac{1+s}{2}\right)^2}{(1-s)^2}=1.
\end{equation}

As in simply asymptotic cases, we first find bounds for the largest possible
horoball type admissible at each asymptotic vertex. Such a horoball is tangent to the facet opposite its center.
Next we set one horoball to be of the largest type, and increase the size of the other horoballs until they become tangent or inadmissible. 
We then vary the types of the horoballs within the allowable range to find the optimal packing density. 
The following lemma was proved in \cite{Sz12} and it gives the relationship between the volumes of two tangent horoball pieces 
centered at certain vertices of a tiling, as we continuously vary their types.

Let $\tau_1$ and $\tau_2$ be two congruent
$n$-dimensional convex pyramid-like regions with vertices at $C_1$ and $C_2$ that share the common edge $\overline{C_1C_2}$.
Let $B_1(x)$ and $B_2(x)$ denote two horoballs centered at $C_1$ and
$C_2$ tangent at point
$I(x)\in {\overline{C_1C_2}}$. Define the point of tangency $I(0)$ (the ``midpoint") such that the
equality $V(0) = 2 vol(B_1(0) \cap \tau_1) = 2 vol(B_2(0) \cap \tau_2)$
holds for the volumes of the horoball sectors. See Figure \ref{fig:2horoballs} (a).

\begin{lemma}[\cite{Sz12}]
\label{lemma:szirmai}
Let $x$ be the hyperbolic distance between $I(0)$ and $I(x)$,
then
\begin{equation}
V(x) = vol(B_1(x) \cap \tau_1) + vol(B_2(x) \cap \tau_2) = \frac{V(0)}{2}\left( e^{(n-1)x}+e^{-(n-1)x}\right) \notag
\end{equation}
strictly increases as $x\rightarrow\pm\infty$.
\end{lemma}

\begin{figure}[h]
\begin{center}
\includegraphics[height=20mm]{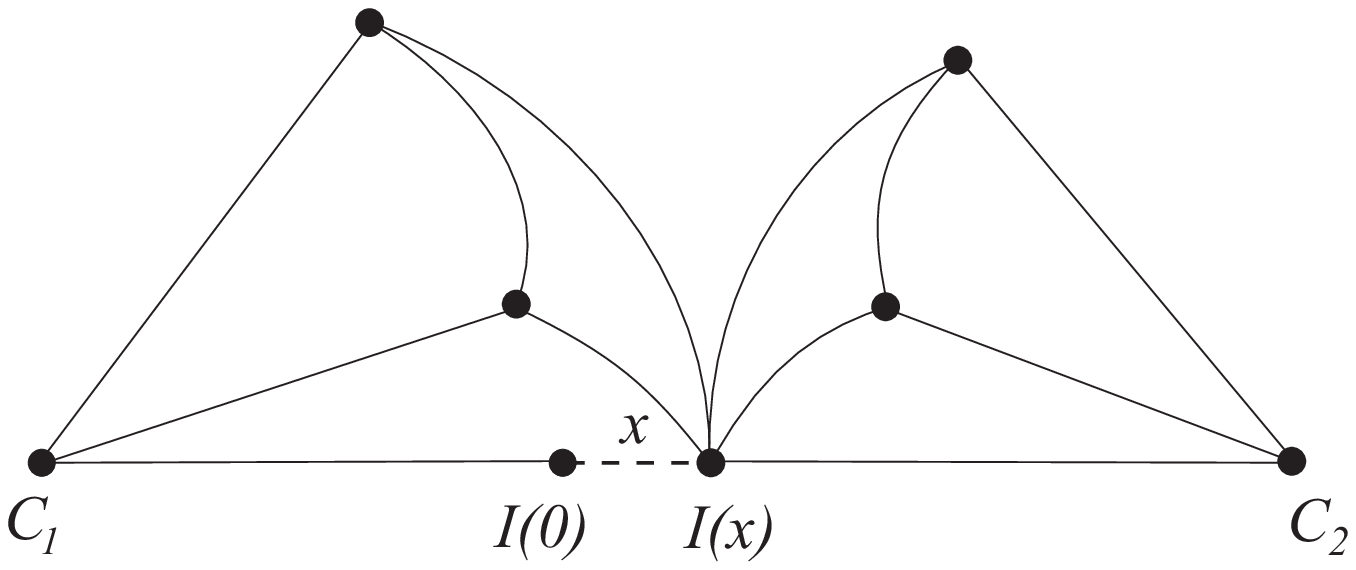} 
\includegraphics[height=40mm]{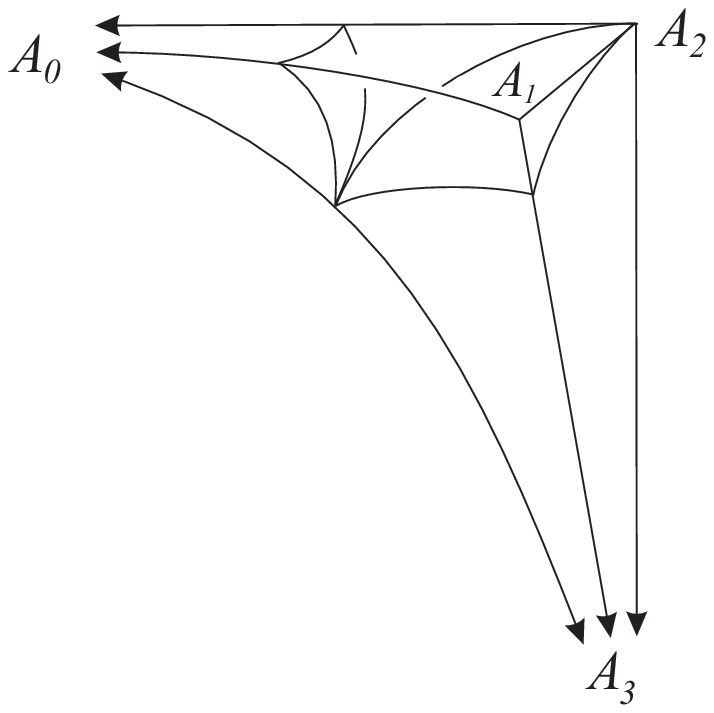}
~~~~~(a)~~~~~~~~~~~~~~~~~~~~~~~~~~~~~~~~~~~~~~~~~~~~~~~~~~~~~~~~~~~~~~(b)
\end{center}
\caption{(a) Lemma 1: two horoballs centered at $C_1$ and $C_2$ tangent at point $l(x)$. (b) Doubly asymptotic case $\mathcal{F}_{\overline{N}_4}$, the horoball types of $\mathcal{B}_0$ and $\mathcal{B}_3$ can be varied between the two limiting cases. The case where $\mathcal{B}_3$ is maximal is pictured.}
\label{fig:2horoballs}
\end{figure} 

\subsubsection{Doubly Asymptotic Case}

\begin{proposition}
The optimal horoball packing density for Coxeter simplex tiling $\cT_{\overline{N}_4}$ is $\delta_{opt}(\overline{N}_4) \approx 0.71644896$.
\end{proposition}

\begin{proof}
Parameterize the fundamental domain $\cF_{\overline{N}_4}$ according to Table \ref{table:data_mult}, so that the two asymptotic vertices are at two opposite poles of the ball model.
Let the corner of the simplex at $A_0$ be $\tau_0$ and that at $A_3$ be $\tau_3$. Place two horoballs $\mathcal{B}_0(s_0)$ and $\mathcal{B}_3(s_3)$ that pass through $(1,0,0,0,s_0)$ and $(1,0,0,0,s_3)$ respectively at $A_0$ and $A_3$.
Let $x_i = \tanh^{-1}(s_i)$ denote the hyperbolic distance of the center of the model $(1,0,0,0,0)$ and the point $S=(1,0,0,0,s_i)$. 
If the two horoballs are of maximal type, $\mathcal{B}_0(0)$ and $\mathcal{B}_3(1/3)$ are tangent to their respective hyperfaces $[\mathbf{u_0}]$ and $[\mathbf{u_3}]$, then their interiors intersect so the packing density is optimal when the horoballs are tangent at one point. Set $s = s_0 = s_3$,
see Figure \ref{fig:2horoballs} (b). Let $\mathcal{B}_i(s) = B_i(x)$, and 
define $V_0(x) = vol(B_0(x) \cap \tau_0)$ and $V_3(x) = vol(B_3(x) \cap \tau_3)$.  
With the techniques of Proposition \ref{proposition:s4} and horosphere equations \eqref{eqn:horosphere1} 
and \eqref{eq:horoY}, we compute that $V_0(0) \approx 0.0138889$ and $V_3(0) \approx 0.00694444$. 
The corner of the simplex at $\tau_1$ is half the size of that at $\tau_3$ so
we have that $2 V_0(0) = V_3(0)$ when $x=s=0$. 
By Lemma \ref{lemma:szirmai}
\begin{equation}
V(x) = V_0(0) e^{3 x}+2 V_0(0)e^{-3 x}
\end{equation}
which is maximal when $x$ is maximized, this happens when $s=1/3$.
\begin{equation}
\delta_{opt}(\overline{N}_4) = \frac{vol(\mathcal{B}_0(1/3) \cap \cF_{\overline{N}_4}) + vol(\mathcal{B}_3(1/3) \cap \cF_{\overline{N}_4})}{vol(\cF_{\overline{N}_4})} \approx 0.71644896. 
\end{equation}
The data for the optimal horoball packing is summarized in Table \ref{table:data_mult}. The symmetry group $\Gamma_{\overline{N}_4}$ carries the density from the fundamental domain to the entire tiling.
\qed
\end{proof}
Similarly to the above proof we obtain
\begin{corollary}
The optimal horoball packing density for Coxeter simplex tiling $\cT_{\overline{BP}_4}$ is $\delta_{opt}(\overline{BP}_4) \approx 0.71644896$.
\end{corollary}
Note that $|\Gamma_{\overline{N}_4}:\Gamma_{\overline{BP}_4}| = 2$ so the tilings are related by fundamental domain doubling. 
The optimal horoball configurations of $\overline{N}_4$ and $\overline{BP}_4$ are essentially the same.
\subsubsection{Triply Asymptotic Case}
We generalize the above results to the two triply asymptotic tilings using the subgroup relations of the multiply 
asymptotic tilings given in Figure \ref{fig:lattice_of_subgroups}. The indices of the subgroups are 
\begin{equation}
|\Gamma_{\overline{N}_4}:\Gamma_{\overline{M}_4}|=
|\Gamma_{\overline{M}_4}:\Gamma_{\overline{DP}_4}|=|\Gamma_{\overline{BP}_4}:\Gamma_{\overline{DP}_4}|=2,
\end{equation}
the fundamental domains are related by domain doubling, hence the optimal packing density is at least $\delta \approx 0.71644896$ for all multiply asymptotic cases. 
By repeated use of Lemma \ref{lemma:szirmai}, we can show that this value is the optimal packing density for all multiply asymptotic 
cases. We omit the technical details of the proof. 
\begin{proposition}
The optimal horoball packing density for triply asymptotic Coxeter simplex tilings $\cT_\Gamma$, $\Gamma \in \Big\{ \overline{M}_4, 
\overline{DP}_4 \Big\}$ is $\delta_{opt}(\Gamma) \approx 0.71644896$.
\end{proposition}

A summary of the results for multiply asymptotic tilings are given in Table \ref{table:data_mult}. 

\begin{table}
	\begin{tabular}{|l|l|l|l|l|}
		 \hline
		 \multicolumn{5}{|c|}{{\bf Coxeter Simplex Tilings} }\\
		 \hline
		 & \multicolumn{2}{|c|}{{\bf Doubly Asymptotic} } & \multicolumn{2}{|c|}{{\bf Triply Asymptotic} }\\
		 \hline
		 Witt Symb. & $\overline{N}_4$ &  $\overline{BP}_4$ &  $\overline{M}_4$ &  $\overline{DP}_4$ \\
		 \hline
		 \multicolumn{5}{|c|}{{\bf Vertices of Simplex} }\\
		 \hline
		 $A_0$ & $(1,0,0,0,1)^*$ & $(1,0,0,0,1)^*$ & $(1,0,0,0,1)^*$ & $(1,0,0,0,1)^*$ \\
		 $A_1$ & $(1,0,\frac{2}{3},0,\frac{1}{3})$ & $(1,0,\frac{2}{3},0,\frac{1}{3})$ & $(1, 0, 1, 0, 0)^*$ & $(1, 0, 1, 0, 0)^*$ \\
		 $A_2$ & $(1,0,\frac{1}{2},\frac{1}{2},0)$ & $(1,-\frac{1}{2},\frac{1}{2},\frac{1}{2},0)$ & $(1,0,\frac{1}{2},\frac{1}{2},0)$ & $(1,-\frac{1}{2},\frac{1}{2},\frac{1}{2},0)$  \\
		 $A_3$ & $(1,0,0,0,-1)^*$ & $(1, 0, 0, 0, -1)^*$ & $(1, 0, 0, 0, -1)^*$ &$(1, 0, 0, 0, -1)^*$  \\
		 $A_4$ & $(1,\frac{1}{2},\frac{1}{2},\frac{1}{2},0)$ & $(1,\frac{1}{2},\frac{1}{2},\frac{1}{2},0)$ & $(1,\frac{1}{2},\frac{1}{2},\frac{1}{2},0)$ & $(1,\frac{1}{2},\frac{1}{2},\frac{1}{2},0)$ \\
		 \hline
		 \multicolumn{5}{|c|}{{\bf The form $\mbox{\boldmath$u$}_i$ of sides opposite $A_i$ }}\\
		\hline
		 $\mbox{\boldmath$u$}_0$ & $(1, 0, -2, 0, 1)^T$ & $(1, 0, -2, 0, 1)^T$ & $(1, 0, -1, -1, 1)^T$ & $(1, 0, -1, -1, 1)^T$ \\
		 $\mbox{\boldmath$u$}_1$ & $(0, 0, 1, -1, 0)^T$ & $(0, 0, 1, -1, 0)^T$ & $(0, 0, 1, -1, 0)^T$ & $(0, 0, 1, -1, 0)^T$  \\
		 $\mbox{\boldmath$u$}_2$ & $(0, -1, 0, 1, 0)^T$ & $(0, -1, 0, 1, 0)^T$ & $(0, -1, 0, 1, 0)^T$ &  $(0, -1, 0, 1, 0)^T$ \\
		 $\mbox{\boldmath$u$}_3$ & $(-1, 0, 1, 1, 1)^T$ & $(-1, 0, 1, 1, 1)^T$ & $(-1, 0, 1, 1, 1)^T$ & $(-1, 0, 1, 1, 1)^T$ \\
		 $\mbox{\boldmath$u$}_4$ & $(0, 1, 0, 0, 0)^T$ & $(0, 1, 0, 1, 0)^T$ & $(0, 1, 0, 0, 0)^T$ & $(0, 1, 0, 1, 0)^T$  \\
		 \hline
		 \multicolumn{5}{|c|}{{\bf Maximal horoball-type parameter $s_i$ for horoball $\mathcal{B}_i$ at $A_i$ }}\\
		\hline
		 $s_0$ & $0$ & $0$ & $-1/3$ & $-1/3$ \\
		 $s_1$ & $-$ & $-$ & $1/3$ & $1/3$ \\
		 $s_3$ & $1/3$ & $1/3$ & $1/3$ & $1/3$ \\
		 \hline
		 \multicolumn{5}{|c|}{ {\bf Volumes of optimal horoball pieces $V_i = vol(\mathcal{B}_i \cap \mathcal{F}_{\Gamma})$}}\\
		\hline
		 $V_0$ & $0.00491046$ & $0.00982093$ & $0.00491046$ & $0.00982093$ \\
		 $V_1$ & $-$ & $-$ & $0.00491046$ & $0.00982093$ \\
		 $V_2$ & $0.0196419$ & $0.0392837$ & $0.0392837$ & $0.0785674$ \\
		\hline
		 \multicolumn{5}{|c|}{ {\bf Optimal Horoball Packing Density} }\\
		\hline
		$\delta_{opt} $ & 0.71644896 & 0.71644896 & 0.71644896 & 0.71644896\\
		\hline
	\end{tabular}
	\caption{Data for multiply asymptotic Coxeter simplex tilings in the Cayley-Klein ball model of radius $1$ centered at $(1,0,0,0,0)$. Vertices marked with $^*$ are ideal.}
	\label{table:data_mult}
\end{table}

\section{Discussion}

In this paper we have concentrated on horoball packings of asymptotic Coxeter simplex tilings of $\mathbb{H}^4$. The main result of this paper is summarized in the following theorem:

\begin{theorem}
\label{mainresult}
In $\mathbb{H}^4$ the horoball packing density $\delta_{opt}(\cT_{\Gamma}) \approx 0.71644896$ 
is optimal in seven asymptotic Coxeter simplex tilings $\Gamma \in \left\{ \overline{S}_4, \overline{P}_4, \widehat{FR}_4, \overline{N}_4, \overline{M}_4, \overline{BP}_4, \overline{DP}_4  \right\}$, 
when horoballs of different types are allowed at each asymptotic vertex of the tiling. 
\end{theorem}

\begin{remark}
Consider two horoball packings to be in a same class if their symmetry groups are isomorphic. 
In this sense one can distinguish between three different horoball packings of optimal density. 
\end{remark}

The optimal packing density obtained in Theorem \ref{mainresult} is the densest ball packing of $\mathbb{H}^4$ known to the authors at the time of writing. It is greater than the value $(5-\sqrt{5})/4 \approx 0.69098301$
 conjectured by L. Fejes T\'oth as the realizable packing density upper bound, pp. 323 \cite{FTL}. 
However, it does not exceed the B\"or\"oczky-type upper bound for $\mathbb{H}^4$ of $0.73046\dots$. The packings we described give a new lower bound for the optimal ball packing density of $\mathbb{H}^4$. 

\begin{corollary}
The optimal ball packing density $\delta_{opt}$ of $\mathbb{H}^4$ is bounded between
$$0.71644896\dots \leq \delta_{opt} \leq 0.73046\dots.$$
\end{corollary}

In this paper we considered the generalized simplicial densities of the horoball packings. It would be instructive to compare to the local Dirichlet--Voronoi densities of each horoball in our family of packings, and present the density of the packing as a weighted average over the cells. Results on the monotonicity of simplicial density function $d_n(r)$ for $n=4$ may help establish the optimality of our packings in $\mathbb{H}^4$ as in the case of $\mathbb{H}^3$ (cf. Section 1). These questions are the subject of ongoing research.



\end{document}